\documentclass[11pt,a4paper]{article}
\usepackage{graphicx}
\usepackage{amsmath,amsfonts,amsthm}
\usepackage{xfrac,esint,mathrsfs,color}
\usepackage[utf8]{inputenc}
\usepackage[colorlinks=true]{hyperref}

\newcommand{\pre}{a}

\newcommand{\alp}{\alpha}

\newcommand\res{\mathop{\hbox{\vrule height 7pt width .5pt depth 0pt
\vrule height .5pt width 6pt depth 0pt}}\nolimits}

\definecolor{green}{rgb}{0,.5,0}

\newcommand\Id{\mathrm{Id}}
\newcommand\sym{\mathrm{sym}}
\newcommand\loc{\mathrm{loc}}

\newcommand\Tr{\mathrm{Tr\,}}

\newcommand\R{\mathbb{R}}

\newcommand\Z{\mathbb{Z}}
\newcommand\calH{\mathcal{H}}

\newcommand\scrS{\mathscr{S}}

\newcommand\calA{\mathcal{A}}

\newcommand\calL{\mathcal{L}}
\newcommand{\LM}[1]{\hbox{\vrule width.2pt \vbox to#1pt{\vfill \hrule width#1pt height.2pt}}}
\newcommand{\LL}{{\mathchoice{\,\LM7\,}{\,\LM7\,}{\,\LM5\,}{\,\LM{3.35}\,}}}

\newtheorem{theorem}{Theorem}[section]

\newtheorem{lemma}[theorem]{Lemma}

\newtheorem{remark}[theorem]{Remark}

\numberwithin{equation}{section}
\newcounter{Nummer}

\usepackage{fancyhdr}
\newcommand{\Hn}{{\mathcal H}^{n-1}}

\newcommand{\Msym}{\mathbb{R}^{n{\times} n}_\sym}

\newcommand{\sumxi}{\sum_{\scriptsize \begin{array}{c}
\xi\in h\Z^n
		\end{array}}}
\begin{document}
\begin{center}
  {\Large Approximation of fracture energies with $p$-growth
  via piecewise affine finite elements}\\[5mm]
{\today}\\[5mm]
Sergio Conti$^{1}$, Matteo Focardi$^{2}$, and Flaviana Iurlano$^{3}$\\[2mm]
{\em $^{1}$
 Institut f\"ur Angewandte Mathematik,
Universit\"at Bonn\\ 53115 Bonn, Germany}\\[1mm]
{\em $^{2}$ DiMaI, Universit\`a di Firenze\\ 50134 Firenze, Italy}\\[1mm]
{\em $^{3}$ Laboratoire Jacques-Louis Lions, Université Paris 6\\ 75005 Paris, France}\\[3mm]
    \begin{minipage}[c]{0.8\textwidth}
    The modeling of fracture problems within geometrically linear elasticity is often based on the space of generalized functions of bounded deformation $GSBD^p(\Omega)$, 
    {$p\in(1,\infty)$,} their treatment is however hindered by the very low regularity of those functions and by the lack of appropriate density results.
We construct here  an approximation of $GSBD^p$ {functions}, {for $p\in(1,\infty)$}, with functions which are Lipschitz continuous away from a jump set which is a finite union of closed subsets of $C^1$ hypersurfaces. 
The strains
of the approximating functions converge strongly {in $L^p$ to the strain of the target,}
and the area of the{ir jump sets 
converge to the area of the target}.
The key idea is to use piecewise affine functions on a suitable grid, which is obtained via the {Freudenthal} partition of a cubic grid.
    \end{minipage}
\end{center}

\section{Introduction}

The modeling of plasticity and fracture in a geometrically linear framework leads to vectorial variational problems 
in which the local energy depends on the symmetric part of the deformation gradient, and the deformation can jump in a set of finite $(n-1)$-dimensional measure
\cite{Suquet1978a,Temam1983,FrancfortMarigo1998}.
If one assumes that the total variation of the distributional symmetrized gradient is controlled by the energy then one deals with functions of bounded deformation,
which are defined as the functions $u\in L^1(\Omega;\R^n)$ such that the distributional strain $Eu:={\frac 12}(Du+Du^T)$ is a bounded measure 
\cite{TemamStrang1980,Temam1983,AnzellottiGiaquinta1980,KohnTemam1983,AmbrosioCosciaDalmaso1997}. Here $\Omega\subset\R^n$ is an open set, and $BD(\Omega)$ denotes the set of functions of bounded deformation on $\Omega$.

In fracture problems one often deals with the {proper subspace} $SBD^p(\Omega)$, which is characterized by the fact that the distributional strain $Eu$ is the sum of an elastic part {$e(u)\calL^n\res\Omega$, with $e(u)\in L^p(\Omega;\R^{n\times n}_\sym)$}, and a {singular part $[u]\odot \nu_u \calH^{n-1}\LL J_u$} concentrated on a {$(n-1)$-rectifiable} set of finite $(n-1)$-dimensional measure,
{with $\nu_u$ the approximate normal and $[u]$ the jump of the traces of $u$ 
across $J_u$} (see \cite{FrancfortMarigo1998,BellettiniCosciaDalmaso1998,Chambolle2003,BourdinFrancfortMarigo2008}).
Typical fracture models, such as Griffith's model, do not, however, give control of the amplitude of the jump of $u$ over the discontinuity set; a typical energy
takes the form 
\begin{equation}\label{eqGriff}
\int_\Omega f(e(u)) dx + \calH^{n-1}(J_u), 
\end{equation}
 which is the natural vectorial generalization of the Mumford-Shah functional {in} linear elasticity. The function $f$ is assumed to be convex and to have $p$-growth at infinity.
One is then lead to compactness results in the space $GSBD^p(\Omega)$, which was introduced by Dal Maso in \cite{gbd} and {recalled} in Section \ref{secprelim} below. 

{In the study of {problems modeled} in $SBD^p$ or $GSBD^p$ {(see for example \cite{hutchinson1989course})} it is crucial to have}
good approximation results {for functions in those spaces}. 
On the one hand, smooth functions are dense in $BD(\Omega)$ in the strict topology, which entails 
weak convergence of the distributional strains. This is clearly not enough to ensure continuity of
{the energy in (\ref{eqGriff}) along such approximating sequences}.  Indeed, the smooth approximants (which are typically obtained by mollification) replace both the discontinuities and the $L^p$ strain $e(u)$ by smooth components, mixing fracture and elastic deformation. It is apparent that this will, in general, increase significantly the energy.
In the scalar case {from the point of view of applications to fracture mechanics,
the functional setting for the problem is provided by $SBV^p(\Omega;\R^n)$ functions, and} 
a density result which guarantees separate convergence of the two terms in (\ref{eqGriff}) was obtained by Cortesani and Toader \cite{CortesaniToader1999}. {The} approximants are still discontinuous, but the jump set has become regular  (a finite union of simplexes) and each function is regular away from the jump set. The gradients converge strongly away from the jump set, and the jumps and the orientation of the jump set converge. 
{More generally, in such a restricted framework one can allow the domain and the codomain to have different dimensions, in what follows we 
shall limit to comment the case of interest in this paper}.

In the vector-valued case, and restricting to energies {with} 
quadratic growth, density of regular functions in ${SBD^2(\Omega)\cap L^2(\Omega;\R^n)}$ was proven by Chambolle in 2004 for $n=2$ \cite{Chambolle2004} and then for $n\ge 3$ \cite{Chambolle2005}. His proof was extended to ${GSBD^2(\Omega)\cap L^2(\Omega;\R^n)}$ by Iurlano \cite{iur12}. Their result shows that any $u\in {GSBD^2(\Omega)\cap L^2(\Omega;\R^n)}$ can be approximated by functions which are 
continuous away from a finite union of closed pieces of $C^1$ hypersurfaces,  are Lipschitz continuous away from this set, with strong convergence of the strains and, in an appropriate sense, of the discontinuities. This permits to obtain convergence of  energies of the type 
(\ref{eqGriff}), as long as $f$ has quadratic growth, and of more general functionals where the surface term has the form $\int_{J_u} \varphi(u^+, u^-, x, \nu)d{\calH^{n-1}(x)}$ for a suitable 
surface energy density $\varphi:\R^n\times \R^n\times\Omega\times S^{n-1}\to\R$. For a discussion of the related problem of density for partition problems we refer to \cite{BraidesContiGarroni2017}.

The restriction of the mentioned results of \cite{Chambolle2004,Chambolle2005,iur12}  to the quadratic energies does not originate from simplicity of presentation, but is instead a consequence of  the type of construction used. 
Indeed, the key idea, first introduced in \cite{Chambolle2004}, is to replace the function $u$ by a componentwise linear approximation on a suitably chosen (very fine) cubic grid, and then to remove the cubes which intersect, in a suitable sense, the jump set. The fact that the energy is quadratic permits an explicit integration of the energy density in each cube, and leads to the identification of the continuum energy of the approximation with a discrete energy, which consists of sums of squares of difference quotients along the edges of the grid. In turn, for a suitable choice of the grid this discrete energy approximates the original energy. For nonquadratic expressions the first step, in which one integrates explicitly over a unit cell, breaks down. Estimates are of course still possible, but the result will only hold up to a $p$-dependent factor, even in the easy case where the functions are smooth to start with.
Therefore we use a different strategy, and resort to a piecewise affine interpolation on a suitable refinement of the grid, see discussion in Section \ref{secmain} below.

Our result permits to replace functions $GSBD^p{(\Omega) \cap L^p(\Omega;\R^n)}$ with much more regular functions in a number of problems related to fracture {(see for example \cite{ContiFocardiIurlano-CRAS,ContiFocardiIurlano-Ex2,CC,ChambolleContiFrancfort2017})}. 
{After the completion of this work, Chambolle and Crismale in \cite{CC} have extended our main result to 
all functions in $GSBD^p(\Omega)$ by adopting a different technique. We stress that the extra integrability hypothesis that we impose 
on the relevant function is often not meaningful for problems in fracture mechanics.} 

{Together with the elliptic regularity results for solutions to 
linear elasticity type systems established in \cite{ContiFocardiIurlanoRegularity},}
 our result has been instrumental for the proof of existence {in dimension $n=2$} of minimizers for {the strong counterpart of} the Griffith functional in (\ref{eqGriff}), that was presented in  \cite{ContiFocardiIurlano-CRAS,ContiFocardiIurlano-Ex2}. 
More precisely, in \cite{ContiFocardiIurlano-Ex2} it is proved that any local minimizer 
$u$ of \eqref{eqGriff} has relatively closed jump set, i.e. 
$\calH^1(\overline{J_u}\cap\Omega\setminus J_u)=0$, and it is smooth outside it, namely 
$u\in C^{1,\alpha}(\Omega\setminus\overline{J_u};\R^2)$ for some $\alpha\in(0,1)$.
The equivalence between the weak formulation of the problem as stated in \eqref{eqGriff}
and the classical strong form then follows (cf. \cite{ContiFocardiIurlano-CRAS,ContiFocardiIurlano-Ex2} for more details).

{Such a mild regularity result extends the analogous statement for $SBV^p$ functions 
proved in the celebrated paper \cite{DegiorgiCarrieroLeaci1989} by De Giorgi, Carriero 
and Leaci, corresponding in applications to the (generalized) antiplane shear setting.   
As already mentioned before, in \cite{ContiFocardiIurlano-CRAS,ContiFocardiIurlano-Ex2} 
the strong approximation property established in this paper is used to infer such 
kind of regularity; viceversa the quoted approximation result by Cortesani and Toader 
\cite{CortesaniToader1999} uses De Giorgi, Carriero and Leaci's regularity result {(in particular by means of 
\cite[Lemma 5.2]{BraidesChiadoPiat} by Braides and Chiadò Piat)} as a key tool to prove the strong approximation property. 

In closing this Introduction we mention a complementary approach to the regularity of $SBD^p$ and $GSBD^p$ functions, which has received a lot of attention in the last years, namely, the proof of rigidity estimates for functions with small jump set. A Korn-Poincar\'e bound in term of the elastic energy alone was proven for 
$SBD^p$ functions in \cite{ChambolleContiFrancfort2016}. An improved estimate, which controls also the gradients, was obtained in the two-dimensional case in \cite{Friedrich1,Friedrich2,ContiFocardiIurlano2017IntRepr}.

{Finally, we summarize the structure of the paper. 
Section \ref{secprelim} is devoted to fix the notation for the piecewise affine finite 
elements and the functional spaces which are involved in our main approximation 
result, Theorem \ref{densitytheorem}, that we shall prove in Section \ref{secmain}.}

\section{Notation}\label{secprelim}
One key ingredient of our piecewise affine approximation is  the {Freudenthal} partition of the $n$-cube $[0,1]^n$. We say that the vertex $(i_1,\dots,i_n)$, $i_k\in\{0,1\}$, precedes the vertex $(j_1,\dots,j_n)$, $j_k\in\{0,1\}$, if $i_k\leq j_k$ for all $k$. The convex hull of a chain of $n+1$ {distinct vertices} is a $n$-simplex. 
Then the {Freudenthal} partition $\scrS$ of the $n$-cube is defined as the set of all $n$-simplexes obtained through maximal chains {of ordered vertices} connecting the origin to the vertex $(1,\dots,1)$ (see Figure 1). Such partition counts $n!$ simplexes.
\begin{figure}[htbp]
\centering
{\includegraphics[height=5cm]{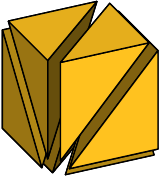}} 
\caption{Freudenthal partition of $[0,1]^3$: the origin is the vertex on the bottom in front, (1,1,1) is at the top in the back.}  
\end{figure}  

Alternatively, for any permutation $\sigma$ of $\{1,\dots, n\}$ one defines a simplex $S_\sigma$ as the convex envelope of the points $v_i:=\sum_{j\le i} e_{\sigma(j)}$, $i=0,1,\dots, n$. Explicitly one obtains
\begin{align*}
S_\sigma:=&\{ \sum_i \lambda_i \sum_{j\le i} e_{\sigma(j)} : \lambda_i\ge 0, \sum\lambda_i=1\}\\
=&\{\sum_j e_{\sigma(j)} \sum_{i\ge j} \lambda_i : \lambda_i\ge 0, \sum\lambda_i=1\}.
\end{align*}
It is then apparent
that $S_\sigma$ consists of the points $x\in[0,1]^n$ such that $j\mapsto x_{\sigma(j)}$ is nonincreasing. Therefore the sets $S_\sigma$ have disjoint interiors and cover $[0,1]^n$. 
{They differ only by a permutation of the components, hence they are congruent and each has volume $1/n!$.}

{We use standard notations for the space $BV$ and its subspaces $SBV^p$
always referring to the book \cite{AmbrosioFuscoPallara} for details.}

{As already mentioned in the introduction $BD(\Omega)$ is the space of functions
$u\in L^1(\Omega;\R^n)$ for which the symmetrized distributional strain 
$Eu=\frac12(Du+Du^T)$ is a Radon measure.
The subspace $SBD^p(\Omega)$, $p\geq 1$, contains all functions $u\in BD(\Omega)$ for which 
\[
Eu=e(u)\calL^n\res\Omega
+(u^+-u^-)\odot\nu_u\calH^{n-1}\res J_u, 
\]
with $e(u)\in L^p(\Omega;\R^{n{\times} n})$ and $\calH^{n-1}(J_u)<\infty$
(cf. \cite{AmbrosioCosciaDalmaso1997,BellettiniCosciaDalmaso1998}). 
}

Given $u\in L^1(\Omega;\R^n)$, for $\Omega\subset\R^n$ open, $\xi\in S^{n-1}$ 
and $y\in\R^n$
one defines the slice $u^\xi_y:\Omega^\xi_y\to\R$ by
${u^\xi_y(t)=u(y+t\xi)\cdot\xi}$, where $\Omega^\xi_y:=\{t\in\R: y+t\xi\in\Omega\}$. If $u\in BD(\Omega)$ one can show that ${u^\xi_y}\in BV(\Omega^\xi_y)$ for almost every $y$.
One denotes with $\Omega^\xi:=(\Id-\xi\otimes \xi)\Omega$ the set of ``relevant'' values of $y$, i.e., the set of $y\in \R^n$ such that $y\cdot\xi=0$ and ${(y+\R\xi)}\cap \Omega\ne\emptyset$.

{An $\mathcal{L}^n$-measurable function $u\colon\Omega\to\mathbb{R}^n$ belongs to $GSBD(\Omega)$ if there exists a bounded positive Radon measure $\lambda_u\in\mathcal{M}_b^+(\Omega)$
such that the following condition holds for every $\xi\in\mathbb{S}^{n-1}$: 
for $\mathcal{H}^{n-1}$-a.e.\ $y\in\Omega^\xi$ {the function {$u^\xi_y(t)=u(y+t\xi)\cdot\xi$}}
belongs to $SBV_\loc(\Omega^\xi_y)$, where $\Omega^\xi_y:=\{t\in\R: y+t\xi\in\Omega\}$, and for every Borel set $B\subset \Omega$ it satisfies
\begin{equation*}
\int_{{\Omega}^\xi}\Big(|Du^\xi_y|(B^\xi_y\setminus J^1_{u^\xi_y})+\mathcal{H}^0(B^\xi_y\cap J^1_{u^\xi_y})\Big)d\mathcal{H}^{n-1}\leq
\lambda_u(B),
\end{equation*}
where $J^1_{u^\xi_y}:=\{t\in J_{u^\xi_y}:|[u^\xi_y](t)|\geq1\}$.}

{If $u\in GSBD(\Omega)$, the aforementioned quantities $e(u)$ and $J_u$ are still well-defined, and are respectively integrable and rectifiable in the previous sense.}
Moreover for every $\xi\in \mathbb{S}^{n-1}$ and for $\Hn$-a.e.\ $y\in \Omega^\xi$ we have
\begin{equation}\label{eqslicepr}
J_{ u^\xi_y}\subset (J_u)^\xi_y\quad\textrm{and}\quad
e(u)(y+t\xi)\xi\cdot\xi=  (u^\xi_y)'(t) \ \text{ a.e.\ $t$ in }\Omega^\xi_y, 
\end{equation}
{where $(u^\xi_y)'$ denotes the absolutely continuous part of the distributional derivative.}
{In analogy to $SBD^p(\Omega)$, the subspace $GSBD^p(\Omega)$ includes all functions in $GSBD(\Omega)$ satisfying $e(u)\in L^p(\Omega;\R^{n{\times} n})$ and $\calH^{n-1}(J_u)<\infty$ (cf. \cite{gbd}).}

\section{The main result}
\label{secmain}
\begin{theorem}\label{densitytheorem}
	Let $\Omega\subset\R^n$ be a bounded Lipschitz set and let {$p>1$.}
	Given $u\in GSBD^p{(\Omega)}\cap L^p(\Omega;\R^n)$,
	there exists a sequence $(u_{{j}})\subset SBV^{{p}}\cap L^\infty(\Omega;\R^n)$ such that each $J_{u_{{j}}}$ is contained in the union $S_{{j}}$ of a finite number of closed connected pieces of  ${C}^1$-hypersurfaces, $u_{{j}}\in W^{1,\infty}(\Omega\setminus S_{{j}};\mathbb{R}^n)$, and the following properties hold:
	\flushleft
	
	\begin{itemize}
		\item[\rm ($1$)] $\|u_{{j}}-u\|_{{L^p(\Omega,\R^n)}}\to0$;

		\item[\rm ($2$)] $\|e(u_{{j}})-e(u)\|_{{L^p(\Omega,\R^{n\times n})}}\to0$;
		
		\item[\rm ($3$)] $\Hn(J_{u_{{j}}})\to\Hn(J_u)$.
	\end{itemize}
\end{theorem}
{\begin{remark}
The sequence $(u_{{j}})$ in Theorem \ref{densitytheorem} can be constructed in a way that it satisfies in addition
\begin{eqnarray*}
&	\Hn(J_{u_j}\triangle J_u)\to0,\\
&\displaystyle\int_{J_{u_j}\cup J_u}\big(|u_j^\pm-u^\pm|\wedge 1\big) d\Hn\to0.	
\end{eqnarray*}	
These further properties can be obtained by following the arguments in \cite{iur12} step-by-step, with obvious modifications due to the fact that the proof there 
is written for $p=2$. In this respect, since only a technical effort is required, we focus here on the main difficulties and we prove Theorem \ref{densitytheorem} in the stated form.
\end{remark}
\begin{remark}By combining Theorem \ref{densitytheorem} and \cite[Theorem 3.1]{CortesaniToader1999} by Cortesani and Toader, 
it is possible to obtain a sequence of approximating functions whose jump set is polyhedral, namely the intersection of $\Omega$ with the union of a finite number of $(n-1)$-dimensional simplexes
compactly contained in $\Omega$. If $p\in(1,2]$ the result can be even improved by taking the $(n-1)$-simplexes pairwise disjoint (see \cite[Remark 3.5]{CortesaniToader1999} and \cite[Section 4, Proof of Corollary 3.11]{cort}).
\end{remark}
}
As mentioned in the Introduction the proof of Theorem \ref{densitytheorem} follows the general strategy of Chambolle and Iurlano \cite{Chambolle2004, Chambolle2005,iur12}, but uses a different interpolation scheme and a different finite-element grid
for the actual construction. {Indeed}, we first construct a sequence of ${SBV^p\cap L^{\infty}}$-functions converging to a given $u\in GSBD^p(\Omega)$, in a way that the bulk estimate is sharp and the surface estimate is obtained up to a 
{multiplicative factor}.
Each approximating function is a piecewise linear interpolation outside from a finite number of cubes, where it is set equal to $0$. Considering piecewise linear interpolations {is essential {in order}
to treat the case {of maps in} $GSBD^p(\Omega)$ with $p\neq 2$. It is the main difference with 
the mentioned references \cite{Chambolle2004, Chambolle2005,iur12} which
deal with the quadratic case $p=2$. Indeed, in \cite{Chambolle2004, Chambolle2005,iur12} 
piecewise polynomial interpolations (of degree equal to the dimension of the space) are employed. 
{S}uch approximations in dimension higher than $3$ if $p\neq 2$ would give rise to a multiplicative factor in the bulk estimate (cf. with \cite[Lemma A.1]{Chambolle2004}), so that the strong approximation 
property would fail.} The piecewise polynomial{s} correspond to a componentwise affine interpolation, that can be done directly on a cubic grid. In the $p\ne 2$ case we need to use a piecewise affine 
interpolation, and therefore need to decompose the domain in simplexes. However, the strategy of  \cite{Chambolle2004, Chambolle2005,iur12} was based on controlling the longitudinal differen{ce} 
quotients along grid segments (the segments joining two vertices of the grid, which are edges of the elements {or diagonals of their faces}). 
{A natural approach would be to} choose an expression for the energy density which uses only these components. In dimension $n=2$ this still works, since one can decompose the square $[0,1]^2$ into two triangles whose sides have the same orientations 
(the three orientations being $(1,0)$, $(0,1)$ and $(1,1)$ for both of the triangles). {In dimension 3 and higher} this is, however, not any more possible, and the energy density will {typically}
not match the geometry of the simplex. Therefore we need to decompose each term of the energy into the components which are ``longitudinal'' with respect to the edges of the simplexes. We shall denote by $A\in\calA\subset\Msym$ the ``components'' of $e(u)$ which enter the energy, and by $\alp^{A,S}_j$ the coefficients of the decomposition of strain direction $A$ in linear combinations of longitudinal strains along the edges simplex $S$, where $j$ labels the sides of $S$. The key observation on which the construction in this paper is based is that one can perform this decomposition jointly for the continuous and for the discrete energy.

We now introduce the objects just mentioned in more detail. We fix $p\geq1$ and choose a finite set of matrices $\calA\subset \Msym$, which span $\Msym$ and are fixed for the rest of the proof.
Let $W:\Msym\to\R$ be defined by
\begin{equation}
\label{e:W}W(\xi):=\sum_{A\in\calA}|\xi :A|^p
\end{equation}
where $A:B:=\Tr A^TB=\sum_{ij}A_{ij}B_{ij}$ denotes the Euclidean scalar product on $\Msym$.
We denote by $D_S$ the set of the edges directions for a simplex {$S$ in the Freudenthal partion $\scrS$}. 
Notice that $D_S$ contains {$\sfrac{n(n+1)}2$ linearly independent} vectors and that {for any given $S$}
the set $\{e\otimes e: e\in D_S\}$ constitutes a basis for $\Msym$.
To see this, it suffices to show that if $\xi\in\Msym$ obeys $e\cdot \xi e=0$ for all $e\in D_S$ then $\xi=0$. 
{Indeed, the simplex
$S$ can be written as the convex envelope of $\{0, f_1, \dots, f_n\}$, where
$(f_i)_{i=1,\dots,n}$ is a basis of $\R^n$. The set $D_S$ is then given by the set of 
the $\pm f_i$'s and the set of all the differences $f_i-f_j$'s with $i\neq j$. Therefore, if {$\xi\in\Msym$ is such that} 
$e\cdot \xi e=0$ for all $e\in D_S$, we deduce first that $f_i\cdot \xi f_i=0$ for all $i$ by taking $e=f_i$, and then 
$f_i\cdot \xi f_{j}=0$ for all $i\neq j$ by taking 
$e=f_i-f_j$. Hence, $\xi=0$.}
We stress that $D_S$ is a set of differences of {vertices} of $S$, not a set of unit vectors.

As in the references mentioned above, the key point is to prove an approximation result that enlarges the jump set by at most a fixed factor. The sharp constant can then be recovered by applying this to the complement of a 
suitable ``large'' compact subset of $J_u$.
\begin{theorem}\label{t:badcost}
	Let $\Omega\subset\R^n$ be an open bounded set with Lipschitz boundary and let $p\geq1$. Given $u\in GSBD^p{(\Omega)}\cap L^p(\Omega;\R^n)$,
	there exists a sequence $(u_{{j}})\subset SBV^{{p}}\cap L^\infty(\Omega;\R^n)$ such that each $J_{u_{{j}}}$ is contained in the union $\Sigma_{{j}}$ of a finite number {of $(n-1)$-dimensional faces of closed simplexes}, $u_{{j}}\in W^{1,\infty}(\Omega\setminus \Sigma_{{j}},\mathbb{R}^n)$, and the following properties hold:
	\flushleft
	
	\begin{itemize}
		\item[\rm ($1$)] $\|u_{{j}}-u\|_{{L^p(\Omega,\R^n)}}\to0$;
		\item[\rm ($2$)] for a positive constant $c_1$ {depending only on $n$ and $p$}
		\[\limsup_{{{j}}\to\infty}\Big(\int_\Omega W(e(u_{{j}}))\,dx+\Hn(\Sigma_{{j}})\Big)\leq\int_\Omega W(e(u))\,dx+c_1\Hn(J_u).
		\]
	\end{itemize}
\end{theorem}
In order to prove Theorem \ref{t:badcost} we need a preliminary lemma, whose proof is entirely similar to \cite[Lemma 3.2]{Chambolle2004}, \cite[Lemma 3]{iur12} {and therefore not repeated here}. The given function $u$ is replaced by another $GSBD^p$-function close in energy to $u$ and defined in a larger set.

\begin{lemma}\label{l:exte}
	Let $\Omega\subset\R^n$, $n\geq2$, be open bounded with Lipschitz boundary, and let $p\geq1$. Given $u\in GSBD^p{(\Omega)}\cap L^p(\Omega;\R^n)$ and $\varepsilon>0$ there exists an open bounded set with Lipschitz boundary $\hat{\Omega}\supset\supset\Omega$ and a function $\hat{u}\in GSBD^p{(\hat{\Omega})}\cap L^p(\hat{\Omega};\mathbb{R}^n)$, such that 
	the following hold 
	\flushleft
	\begin{itemize}
		\item[\rm ($1$)] $\displaystyle ||\hat{u}-u||_{L^p(\Omega,\R^n)}<\varepsilon$,	
		\item[\rm ($2$)] $\displaystyle\int_{\hat{\Omega}}|e(\hat{u})|^p \,dx\leq \int_{\Omega}|e(u)|^p\,dx+\varepsilon$,	
		\item[\rm ($3$)] $\displaystyle\mathcal{H}^{n-1}(J_{\hat{u}})\leq \mathcal{H}^{n-1}(J_{u})+\varepsilon$.	
		\end{itemize} 
		\end{lemma}
\begin{proof}[Proof of Theorem \ref{t:badcost}]
Fixed $u\in GSBD^p(\Omega)\cap L^p(\Omega;\R^n)$ and $\varepsilon>0$, Lemma \ref{l:exte} provides $\hat{u}$ and $\hat{\Omega}$ satisfying (1)-(3).

Fixed $y\in [0,1)^n$ and $h>0$ small, we consider the {translated} lattice $hy+\xi$, with $\xi\in h\mathbb{Z}^n$.
We introduce 
{the tubular neighborhood in the direction $-\tau$}
of $J_{\hat{u}}$, 
\[
	 J^\tau:=\bigcup_{x\in J_{\hat{u}}}[x,x-\tau]
	 =\{y\in\R^n: [y,y+\tau]\cap J_{\hat u}\ne\emptyset\},\quad \textrm{for }\tau\in\R^n	 ,
\]
and the {longitudinal} difference quotient along the edge $\tilde e_j$ of $S\in \scrS$
\[
\triangle^S_{j,h}(z):=\frac{(\hat u(z+h\pre_j+h\tilde e_j)-\hat u(z+h\pre_j))\cdot \tilde e_j}{h|\tilde e_j|^2}
\]
for $[z+h\pre_j,z+h\pre_j+h\tilde e_j]\subset\hat\Omega$, 
and zero elsewhere,
where $\pre_j$ and $\pre_j+\tilde e_j$ are the only two vertices of {$S$} whose difference is $\tilde e_j$.
Let us introduce the discrete bulk and surface energies 
\begin{align}\label{e:E2}
& \displaystyle E_1^{y,h}(\hat{\Omega}):=\frac{h^n}{n!}
\sum_{\scriptsize \begin{array}{c}
	 {A\!\!\in\!\!\calA}\\ S\!\!\in\!\! \scrS
	\end{array}}\sumxi
 \nonumber\\
& \hskip3cm
\Big|\sum_j\alp^{A,S}_{j}\triangle^S_{j,h}(\xi+hy)(1-1_{J^{h\tilde e_j}}(hy+\xi+ha_j))\Big|^p,\nonumber\\
& \displaystyle E_2^{y,h}(\hat{\Omega}):=\tilde{c}_1 h^{n-1}\sum_{e\in \bigcup_S D_S}\sumxi \frac{1_{J^{he}}(hy+\xi)}{|e|},
\end{align}
where {$1_B$ denotes the characteristic function of the set $B$}, $\alp^{A,S}_j$ are the coordinates of {$A$} in the basis {$\{\tilde \nu_j\otimes \tilde \nu_j:\ \tilde e_j\in D_S\}$ of $\Msym$, where $\tilde \nu_j=\tilde e_j/|\tilde e_j|$} 
and {$\tilde c_1:=2^nn\sqrt n$, the latter choice will be motivated later}.

Let $w_{y,h}$ be the piecewise {affine} function obtained interpolating $\hat{u}$ on each simplex of the partition.
Let us prove that there exist $y\in[0,1)^n$ and a subsequence of ${h\downarrow 0}$ not relabeled, such that 
\begin{itemize}
	\item[\rm ($1'$)] $\displaystyle\|w_{y,h}-\hat{u}\|_{L^p(\Omega,\R^n)}\to0$;
	\item[\rm ($2'$)] $\displaystyle\lim_{h\to\infty}\Big[E^{y,h}_1(\hat{\Omega})+ E^{y,h}_2(\hat{\Omega})\Big]\leq\int_{\hat{\Omega}} W(e(\hat{u}))dx+c_1\Hn(J_{\hat{u}})$,
	where $c_1$ is a constant depending on $\tilde{c}_1$ and $W$ is {the integrand}
	defined in \eqref{e:W}.
\end{itemize}
In order to prove ($1'$) we observe that for every simplex $\xi+hy+hS$ of the partition with vertices $a_i$, $i=0,\dots,n$, there exist $n+1$ affine functions $f_i$ such that 
\begin{equation*}\sum_{i=0}^{n}f_i=1\hskip5mm\text{ and } \hskip5mm
 w_{y,h}=\sum_{i=0}^{n}\hat{u}(a_i)f_i \quad\textrm{on } \xi+hy+hS.
\end{equation*}
Then integrating on $[0,1)^n$, we deduce by convexity and Fubini's theorem {for $h$ sufficiently small}
\begin{align*}
	\int_{[0,1)^n}dy &\int_{\Omega}|w_{y,h}{(x)}-\hat{u}{(x)}|^pdx\\
	& \leq
	c\int_{[0,1)^n}dy\sum_{\xi\in h\Z^n\cap\hat{\Omega}}\int_{\Omega\cap(\xi+hy+[-h,h]^n)}|\hat{u}(\xi+hy)-\hat{u}{(x)}|^pdx\\
	& =c\int_{\Omega}dx\sum_{\xi\in h\Z^n\cap\hat{\Omega}}	\int_{[0,1)^n}1_{\xi+hy+[-h,h]^n}(x)|\hat{u}(\xi+hy)-\hat{u}(x)|^pdy,
\end{align*}	
{where $c$ takes into account the convexity of the power 
$\R\ni t\mapsto |t|^p$ and the number of simplexes sharing a 
certain $a_i$ as a vertex.}
Changing variable $z=(x-hy-\xi)/h$ in the second integral we obtain
\begin{align*}
\int_{[0,1)^n}dy &\int_{\Omega}|w_{y,h}{(x)}-\hat{u}{(x)}|^pdx\\ 
&\leq c\int_{\Omega}dx\sum_{\xi\in h\Z^n\cap\hat{\Omega}}
\int_{(\frac{x-\xi}{h}-[0,1)^n)\cap(-1,1)^n}|\hat{u}(x-hz)-\hat{u}(x)|^pdz\\ 
& \leq c\int_{(-1,1)^n}dz\int_{\Omega}|\hat{u}(x-hz)-\hat{u}(x)|^pdx.
\end{align*}	
By dominated convergence theorem the last term vanishes {as $h\downarrow 0$}, 
therefore there is a subsequence of $h$, not relabeled, and a measurable set of full 
measure $E\subset[0,1)^n$ such that for every $y\in E$ the convergence in $(1')$ holds.

Let us prove now $(2')$. We integrate again on $[0,1)^n$ and estimate first the bulk part.
Changing variable $x=hy+\xi$ and then slicing through Fubini's theorem we obtain
\begin{align}\label{E3amezza}
 \int_{[0,1)^n}& E_1^{y,h}(\hat{\Omega})\,dy\notag\\
&=\frac{1}{n!}
\sum_{\scriptsize \begin{array}{c}
	 {A\!\!\in\!\!\calA}\\ S\!\!\in\!\! \scrS
	\end{array}}
\sum_{\xi\in h\Z^n}\int_{\xi+h[0,1)^n}
\nonumber\\
&\hskip3cm
{{1_{\hat{\Omega}}(x)}} \Big|\sum_j\alp^{A,S}_{j}\triangle^S_{j,h}(x)(1-1_{J^{h\tilde e_j}}(x+ha_j))\Big|^pdx\notag\\
&=\frac{1}{n!}
\sum_{\scriptsize \begin{array}{c}
	 {A\!\!\in\!\!\calA}\\ S\!\!\in\!\! \scrS
	\end{array}}\int_{{\hat{\Omega}}}
\Big|\sum_j
\frac{\alp^{A,S}_{j}}{h|\tilde e_j|}\Big(\hat{u}^{\tilde \nu_j}_{\pi_j(x+ha_j)}((x+ha_j)\cdot\tilde \nu_j+h|\tilde e_j|)-\notag\\
&\hat{u}^{\tilde \nu_j}_{\pi_j(x+ha_j)}((x+ha_j)\cdot\tilde \nu_j)\Big)(1-1_{J^{h\tilde e_j}}(x+ha_j))\Big|^pdx,
\end{align}
where $\pi_j$ is the {orthogonal} projection on $\Pi^{\tilde \nu_j}{:=\tilde \nu_j^\perp}$.
We {recall that} $\tilde \nu_j=\tilde e_j/|\tilde e_j|$ and that the slice is defined as usual by {$\hat{u}^{\nu}_z(s):=\hat{u}(z+s\nu)\cdot \nu$ for $z\in\Pi^{\nu}$.}

Since $\hat{u}\in GSBD^p{({\hat{\Omega}})}\cap L^p(\hat{\Omega};\R^n)$ we have 
{$\hat{u}^{\nu}_z\in SBV^p(\hat{\Omega}^{\nu}_z)$, for {$\Hn$}-a.e. $z\in\Pi^{\nu}$. Observe that $1_{J^{he}}(z+s\nu)=0$, for $e=|e|\nu$ and $z\cdot\nu=0$, 
means $z+s\nu\not\in J^{he}$, which is the same as $[s, s+h|e|]\cap (J_{\hat u})_z^\nu=\emptyset$. For almost every $z$, by (\ref{eqslicepr}) this implies
$[s,s+h|e|]\cap J_{\hat u_z^\nu}=\emptyset$. Therefore}
\eqref{E3amezza} yields 
\begin{align}\label{E3amezza2}
&\int_{[0,1)^n}  E_1^{y,h}(\hat{\Omega})\,dy\nonumber\\
& \leq \frac{1}{n!}
\sum_{\scriptsize \begin{array}{c}
	 {A\!\!\in\!\!\calA}\\ S\!\!\in\!\! \scrS
	\end{array}}\int_{{\hat{\Omega}}}
\Big|\sum_j\frac{\alp^{A,S}_{j}}{h|\tilde e_j|}\int_0^{h|\tilde e_j|}(\hat{u}^{\tilde \nu_j}_{\pi_j(x+ha_j)})'(t+(x+ha_j)\cdot\tilde \nu_j))dt\Big|^pdx\nonumber\\
&=\frac{1}{n!}
\sum_{\scriptsize \begin{array}{c}
	 {A\!\!\in\!\!\calA}\\ S\!\!\in\!\! \scrS
	\end{array}}\int_{{\hat{\Omega}}}
\Big|\sum_j\alp^{A,S}_{j}\int_0^{h|\tilde e_j|}\hspace{-1.05cm}-(e(\hat u)\tilde \nu_j\cdot\tilde \nu_j)(x+ha_j+t\tilde \nu_j)dt\Big|^pdx.
\end{align}
As ${h\downarrow 0}$ we find by the continuity of the translation, the Lebesgue theorem, and the dominated convergence theorem 
\begin{align*}
\limsup_{h\downarrow 0}&\int_{[0,1)^n}E_1^{y,h}(\hat{\Omega})\,dy\leq
\frac{1}{n!}
\sum_{\scriptsize \begin{array}{c}
	 {A\!\!\in\!\!\calA}\\ S\!\!\in\!\! \scrS
	\end{array}}
\int_{\hat{\Omega}}
\Big|\sum_j\alp^{A,S}_{j}e(\hat u)\tilde \nu_j\cdot\tilde \nu_j\Big|^pdx\notag\\
&=\frac{1}{n!}
\sum_{\scriptsize \begin{array}{c}
	 {A\!\!\in\!\!\calA}\\ S\!\!\in\!\! \scrS
	\end{array}}
\int_{\hat{\Omega}}
\Big|\sum_j\alp^{A,S}_{j}e(\hat u):\tilde \nu_j\otimes\tilde \nu_j\Big|^pdx\notag\\
& =\frac{1}{n!}
\sum_{\scriptsize \begin{array}{c}
	 {A\!\!\in\!\!\calA}\\ S\!\!\in\!\! \scrS
	\end{array}}
\int_{\hat{\Omega}}
\Big|e(\hat u):{A}\Big|^pdx
= \int_{\hat{\Omega}}W(e(\hat u))dx.
\end{align*}
Arguing in a similar way for $E^{y,h}_2$ we obtain
\begin{align}\nonumber
\displaystyle\int_{{[0,1)^n}}E_2^{y,h}(\hat{\Omega})\,dy
&{=\tilde c_1 
\sum_{e\in \cup_{S}D_S}
\int_{\R^n} \frac{1_{J^{he}}}{h|e|} dz}\\
&\leq \tilde{c}_1\sum_{e\in \cup_{S}D_S}\int_{J_{\hat{u}}}|\nu_{\hat{u}}\cdot \nu_e|d{\Hn}\leq c_1 {\Hn}(J_{\hat{u}}),
\label{E2a}
\end{align}
having set {$c_1:=\tilde{c}_1\#(\cup_{S\in\scrS}D_S)=
2^{n-1}n^{\sfrac 52}(n+1)n!$, $\nu_e:=e/|e|$, and having used 
in the last inequality the slicing formula
$$
\int_{J_{\hat{u}}}|\nu_{\hat{u}}\cdot\nu_e|\,d\Hn=
\int_{\Pi^{\nu_e}}\#(J_{\hat u^{\nu_e}_z})\,d\Hn(z)\,.
$$
}

By inequalities \eqref{E3amezza2} and \eqref{E2a} and Fatou's lemma we conclude there there exists $y\in [0,1)^n$ and a subsequence of ${h\downarrow 0}$ not relabeled {for convenience} such that properties ($1'$) and ($2'$) hold. {In what follows we drop the index $y$ ad denote $w_{y,h}$ simply by $w_h$}.

\medskip
We define now {a sequence} $v_h$ as $0$ in the cubes $Q=\xi+hy+[0,h)^n$ such that 
$J_{\hat{u}}$ crosses an edge of {$\xi+hy+hS$} for some $S\in \scrS$,
while we set $v_h:=w_h$ otherwise. In the first case we say that the cube is bad, in the second 
case that it is good. {We let $\Sigma_h$ be the union of the faces of the bad cubes.} We claim that
\begin{itemize}
	\item[\rm ($1''$)\phantom{a.}] $\displaystyle\|w_h-v_h\|_{L^p(\Omega,\R^n)}\to0$,
	\item[\rm ($2''$)\phantom{a.}] the constant $\tilde{c}_1$ in (\ref{e:E2}) can be chosen 
	in a way that {for every $h$ sufficiently small}\\
	$\displaystyle\int_{\Omega}W(e(v_h))dx+{\Hn(\Sigma_h)}\leq E^{h}_1(\hat{\Omega})+ E^{h}_2(\hat{\Omega})$.
\end{itemize}
As for ($2''$), we first notice that $\Hn(\overline{J_{v_h}})\leq 2nh^{n-1}N_h$, being $N_h$ the number of bad cubes. 
{For every bad cube $Q:=\tilde\xi+hy+[0,h)^n$ there is at least one pair $e\in \bigcup_S D_S$, $\xi\in h\Z^n$, such that
$[\xi+hy,\xi+hy+he]\subset\overline Q$ and $1_{J^{he}}(\xi+hy)=1$. At the same time, the edge 
$[\xi+hy,\xi+hy+he]$ is shared by at most $2^{n-1}$ cubes. Therefore
$$
N_h\le 2^{n-1} \sum_{e\in \bigcup_S D_S}\sum_{\xi\in h\Z^n} 1_{J^{he}}(\xi+hy).$$
Recalling that $|e|\le \sqrt n$ and defining $\tilde c_1:=2^nn\sqrt n$ we obtain from the definition of 
$E^h_2(\tilde \Omega)$ that
\begin{equation}\label{eq:sup}
\calH^{n-1}(\Sigma_h)\le 2n h^{n-1} N_h\leq E^h_2(\tilde \Omega).
\end{equation}}
Let us prove now that
\begin{equation}\label{eq:vol}
\int_{\Omega}W(e(v_h))dx\leq E^h_1(\hat{\Omega}).
\end{equation}
By definition of $W$, for each good cube $Q=[0,h)^n+\xi+hy$ of the lattice being $v_h=w_h$ we obtain
\[\int_{[0,h)^n+\xi+hy}W(e(v_h))dx=\sum_{\scriptsize \begin{array}{c} {A\!\!\in\!\!\calA}\\ S\!\!\in\!\!\scrS \end{array}}\int_{hS+\xi+hy}|e(w_h){:A}|^pdx,\]
since $\scrS$ gives a partition of the cube. Recalling that $\alpha_j^{A,S}$ denotes the coefficients of {$A$} in the basis $\{\tilde \nu_j\otimes \tilde \nu_j:\ \tilde e_j\in D_S\}$, we have 
\begin{align*}
\int_{[0,h)^n+\xi+hy} W(e(v_h))&dx
=\sum_{\scriptsize \begin{array}{c} {A\!\!\in\!\!\calA}\\ S\!\!\in\!\!\scrS \end{array}}
\int_{hS+\xi+hy}|e(w_h):\sum_{j}\alpha_j^{A,S}
\tilde \nu_j\otimes\tilde \nu_j|^pdx\notag\\
&=\sum_{\scriptsize \begin{array}{c} {A\!\!\in\!\!\calA}\\ S\!\!\in\!\!\scrS \end{array}}
\int_{hS+\xi+hy}|\sum_{j}\alpha_j^{A,S} e(w_h)\tilde \nu_j\cdot\tilde \nu_j|^pdx.
\end{align*}
Since $w_h$ is the affine interpolation of ${\hat u}$ on each simplex constituting $Q$, we have
\[e(w_h)\nu\cdot\nu=\frac{(w_h(a)-w_h(b))\cdot\nu}{|a-b|}=
{\frac{(\hat u(a)-\hat u(b))\cdot\nu}{|a-b|}},\]
for every pair $a,b$ of vertices of $Q$, with $\nu=(a-b)/|a-b|$.
Therefore 
\[\int_{[0,h)^n+\xi+hy}W(e(v_h))dx= 
\sum_{\scriptsize \begin{array}{c} {A\!\!\in\!\!\calA}\\ S\!\!\in\!\!\scrS \end{array}}
\frac{h^n}{n!} \Big|\sum_j\alp^{A,S}_{j}\triangle^S_{j,h}(\xi+hy)\Big|^p,\]
recalling that the difference quotient is defined by
\[
\triangle^S_{j,h}(z)=\frac{(\hat u(z+h\pre_j+h\tilde e_j)-\hat u(z+h\pre_j))\cdot \tilde e_j}{h|\tilde e_j|^2},\quad \textrm{for }z\in \hat\Omega,
\]
and that {$\pre_j,\pre_j+e_j$} represent the only two vertices of $S$ whose difference is {$\tilde e_j$}. Summing on the good cubes $Q$ which intersect $\Omega$ we finally obtain \eqref{eq:vol}. Property ($2''$) then follows by \eqref{eq:sup} and \eqref{eq:vol}. 

To check ($1''$) we use ($1'$) and we observe that
\[\|v_h-w_h\|^p_{L^p(\Omega,\R^n)}=\int_{{\mathscr{C}}}|w_h|^pdx,\]
where {$\mathscr{C}$ denotes} the union of the bad cubes. Notice that {$\mathscr{C}$ has
small Lebesgue measure, indeed}
{by (\ref{eq:sup})
\[\calL^n(\mathscr{C})\leq h^n N_h =O(h).\]}

Properties ($1'$), ($2'$), ($1''$), ($2''$), together with Lemma \ref{l:exte}, imply ($1$) and ($2$).
\end{proof}
Using a {by now standard} Besicovitch covering argument we can refine the estimate obtained in Theorem \ref{t:badcost} reducing  the coefficient of the surface term to $1$. {The idea is to cover the most of the jump set of $u$ with a finite number of pairwise disjoint closed balls, in a way that the jump set in each of them is close to a $C^1$ hypersurface separating the ball into two components. Then Theorem \ref{t:badcost} is applied in each component and in the complement of the balls, so that the jump of the resulting function is the union of the $C^1$ hypersurfaces separating the balls and of the $(n-1)$-dimensional faces of closed simplexes obtained by applying Theorem \ref{t:badcost} (see \cite[Theorem 2]{Chambolle2004} or \cite[Theorem 6]{iur12} for more details)}.
\begin{theorem}\label{t:goodcost}
	Let $\Omega\subset\R^n$ be an open bounded set with Lipschitz boundary, {$p\in(1,\infty)$,} and let $u\in GSBD^p{(\Omega)}\cap L^p(\Omega;\R^n)$. Then there exists a sequence $u_{{j}}\in SBV^{{p}}\cap L^p(\Omega;{\R^n})$ 
	such that $J_{u_{{j}}}$ is contained in the union $S_{{j}}$ of a finite number of closed connected pieces of ${C}^1$-hypersurfaces, $u_{{j}}\in W^{1,\infty}(\Omega\setminus S_{{j}};{\R^n})$, and the following properties hold:
	\flushleft
	
	\begin{itemize}
		\item[\rm ($1$)] $\|u_{{j}}-u\|_{{L^p(\Omega;\R^n)}}\to0$;
		\item[\rm ($2$)] $\displaystyle\limsup_{{{j}}\to\infty}\Big(\int_\Omega W(e(u_{{j}}))\,dx+\mathcal{H}^{{n-1}}(S_{{j}})\Big)\leq\int_\Omega W(e(u))\,dx+\mathcal{H}^{{n-1}}(J_u)$.
		\end{itemize}
		\end{theorem}
{\begin{remark}
We emphasize that the growth hypothesis $p>1$ is in fact not needed in the proof of Theorem \ref{t:goodcost}. It is only used to deduce 
conditions (2) and (3) in Theorem \ref{densitytheorem} by means of
the $GSBD^p$ compactness result in \cite[Theorem 11.3]{gbd} 
and strict convexity of $W$ as oulined below.
\end{remark}}	
Theorem \ref{densitytheorem} easily follows from Theorem \ref{t:goodcost}.		
\begin{proof}[Proof of Theorem \ref{densitytheorem}.]
	Let $u_{{j}}$ be given by Theorem \ref{t:goodcost}. By compactness in {$GSBD^p$} \cite[Theorem 11.3]{gbd} there exists a subsequence of $u_{{j}}$, not relabeled, satisfying
	\begin{eqnarray}
	\label{e}
	& e(u_{{j}})\rightharpoonup e(u)\quad \textrm{weakly in }L^p(\Omega,\R^{{n{\times}n}}),\\
	\label{q}
	& \displaystyle\int_{\Omega}W(e(u))dx\leq\liminf_{{{j}}\to\infty}\int_{\Omega}W(e(u_h))dx,\\
	\label{h}
	& \mathcal{H}^{{n-1}}(J_u)\leq\liminf_{{{j}}\to\infty}\mathcal{H}^{{n-1}}(J_{u_{{j}}}).
	\end{eqnarray}
	By {virtue of} inequality ($2$) of Theorem \ref{t:goodcost}, (\ref{q}), and (\ref{h}) we obtain
	\begin{eqnarray*}
	& \displaystyle\int_{\Omega}W(e(u))dx=\lim_{{{j}}\to\infty}\int_{\Omega}W(e(u_{{j}}))dx,\\
	& \displaystyle\mathcal{H}^{{n-1}}(J_u)=\lim_{{{j}}\to\infty}\mathcal{H}^{{n-1}}(J_{u_{{j}}}),
	\end{eqnarray*}
and the thesis follows at once {from \eqref{e} and the strict convexity of $W$}.
\end{proof}		
\section*{Acknowledgments} 

This work was partially supported 
by the Deutsche Forschungsgemeinschaft through the Sonderforschungsbereich 1060 
{\sl ``The mathematics of emergent effects''}, project A5. 
{F.~Iurlano has been partially supported by the program FSMP and the project Emergence Sorbonne-Université ANIS.}
{M. Focardi has been partially supported by GNAMPA.}


\end{document}